\documentclass[10pt, a4paper]{article}

\RequirePackage{amssymb,amsthm,amsmath}
\usepackage[colorlinks=true,linkcolor=blue, citecolor=red]{hyperref}
\usepackage[nameinlink,capitalize]{cleveref}
\usepackage{enumitem}
\usepackage{titlesec}
\usepackage[dvipsnames]{xcolor}
\usepackage{amsfonts,amsmath,amssymb,amsthm,dsfont,bbm}
\crefname{equation}{}{} 
\numberwithin{equation}{section}

\titleformat*{\section}{\large\bfseries}
\titleformat*{\subsection}{\bfseries}

\usepackage{hyperref}

\usepackage[mathscr]{euscript}
\usepackage[font=footnotesize,labelfont=bf]{caption}
\usepackage{a4wide}
\usepackage{cancel}

\theoremstyle{plain}
\newtheorem{theorem}{Theorem}

\newtheorem{lemma}{Lemma}
\newtheorem{proposition}{Proposition}

\theoremstyle{definition}
\newtheorem{definition}{Definition}
\newtheorem{example}{Example}
\newtheorem{remark}{Remark}

\newtheorem{assumption}{Assumption} 

\newcommand{\od}{\mathrm{d}}

\newcommand{\E}{{\mathbb{E}}}
\renewcommand{\P}{{\mathbb{P}}}
\newcommand{\R}{{\mathbb{R}}}
\newcommand{\N}{{\mathbb{N}}}

\newcommand{\F}{{\mathbb{F}}}

\newcommand{\cB}{\mathcal{B}}

\newcommand{\cF}{\mathcal{F}}

\newcommand{\cC}{\mathcal{C}}

\newcommand{\cN}{\mathcal{N}}

\newcommand{\law}[1]{\mathbb{P}_{#1}}

\newcommand{\ep}{\varepsilon}
\newcommand{\ind}{\mathbbm{1}}
\newcommand{\rplus}{[0, \infty)}

\newcommand{\tmax}{T_{\max}}

\newcommand{\lognormal}{\mathrm{LogNormal}}

\newcommand{\imean}{\bar{\mu}_0}
\newcommand{\ivar}{\bar{\sigma}_0}

\newcommand{\variance}{\mathrm{Var}}

\newcommand{\corr}[1]{{#1}}
\newcommand{\new}[1]{{#1}}
\newcommand{\newer}[1]{{#1}}
\newcommand{\newest}[1]{{#1}}
\newcommand{\changed}[1]{{#1}}
\newcommand{\changedagain}[1]{{#1}}

\allowdisplaybreaks
\numberwithin{equation}{section}

\begin{document}

\title{\large \bf{Mean field stochastic differential equations with a diffusion coefficient with irregular distributional dependence}}

\author{Jani Nykänen$^1$}

\date{}
\maketitle

\begin{abstract}
We study mean field stochastic differential equations with a diffusion coefficient that depends on the distribution function of the unknown process in a discontinuous manner, \corr{which is a type of distribution dependent regime switching.} To determine the distribution function we show \corr{that} under certain conditions these equations can be transformed into SDEs with deterministic coefficients \corr{using a Lamperti-type transformation}. We prove an existence and uniqueness result and consider cases when the uniqueness may fail or a solution exists only for a finite time.
\end{abstract}

{\small{
\vspace{1em}
{\noindent \textbf{Keywords:} mean field stochastic differential equation, discontinuous diffusion coefficient, \newest{Lamperti-type} transformation, existence and non-existence of strong solutions, regime switching  
}

\tableofcontents

{\noindent
\footnotetext[1]{Department of Mathematics and Statistics, University of
Jyv\"askyl\"a, Finland. \\ \hspace*{1.5em}
  {\tt jani.m.nykanen@jyu.fi}}
}
}}


\section{Introduction}

We study one-dimensional mean field stochastic differential equations of the type 
\begin{equation} \label{eq_very_general_equation}
X_t = x_0 + \int_0^t \sum_{n=1}^\infty \ind_{\left\{ \P(X_s \leq r(s)) \in I_n \right\}} \sigma_n(s, X_s, \law{X_s}) \od B_s + \int_0^t b(s, X_s, \law{X_s}) \od s, \quad t \in \rplus,
\end{equation}
where $B$ is a one-dimensional Brownian motion, $x_0$ is an initial condition independent of $B$, $(I_n)_{n = 1}^\infty$ are pairwise disjoint Borel sets forming a partition of $[0, 1)$ and $r : \rplus \to \R$ is a \corr{Borel} measurable function. The diffusion coefficient function is then 
\begin{equation} \label{eq_type_of_disc}
\sigma(t, x, \mu) = \sum_{n=1}^\infty \ind_{\left\{ \mu((-\infty, r(t)]) \in I_n \right\}} \sigma_n(t, x, \mu),
\end{equation}
where $\mu$ is a Borel probability measure on $\R$. \corr{This is a type of distribution dependent regime switching where the diffusion coefficient is switched to \new{another} one when a deterministic function depending on the law of the unknown process attains certain values. }

These equations are comparable to the equations studied in \cite{paper1}, where the diffusion coefficient function is of the type
\begin{equation} \label{eq_paper1_sde}
\sigma(t, x, \mu) = \sum_{n=1}^\infty \ind_{\left\{ \int_{\R} \left| y - z \right|^p \od \mu(y) \in I_n  \right\}} \sigma_n(t, x, \mu)
\end{equation}
for some fixed $z \in \R$ and exponent $p \geq 2$. \newer{The moment function $t \mapsto \E \left| X_t - z \right|^p$ can be computed by applying Itô's formula to the function $x \mapsto \left| x - z \right|^p$ and then taking the expectation. However, since the map $x \mapsto \ind_{\left\{ x \in (-\infty, r(t)] \right\}}$ is more irregular applying Itô's formula in the same way as in \cite{paper1} is not possible, so to determine the function $t \mapsto \P(X_t \leq r(t))$ in \eqref{eq_very_general_equation} one needs an alternative method.
}

Our approach is to transform \eqref{eq_very_general_equation} into an SDE with deterministic coefficients, because its distribution is \corr{Gaussian and hence} well known. This is made possible by the following observation: supposing that we have a \newer{continuous} transformation $F : \newest{D}  \to \R$, \newest{where the domain $D$ is defined in \eqref{eq_def_D},} that is \newer{strictly monotone} in the second variable, then 
\begin{equation*}
\P(X_t \leq r(t)) = \P(F(t, X_t) \leq F(t, r(t)))
\end{equation*}  
for all $t \in \rplus$. Using the transformation we are able to study similar effects as in \cite{paper1}: we prove an existence and uniqueness theorem, which relies on the monotonicity of the function $t \mapsto \P(X_t \leq r(t))$. \corr{Moreover, we} consider cases when there exists more than one solution and when the existence of a (global) solution fails due to an oscillation effect.

\bigskip

To motivate the equation \eqref{eq_very_general_equation} one can think the SDE \eqref{eq_very_general_equation} as a mean field limit of a large particle system \corr{(supposing first that the diffusion coefficient is approximated with more regular functions)}. Under this interpretation the quantity $\P(X_t \leq r(t))$ expresses the relative amount of "particles" that are below a given reference rate. One can think, for example, asset prices: when the percentage of the assets that are priced below a given rate function hits a given threshold, the magnitude of the volatility changes in a discontinuous manner. 

\bigskip

Stochastic differential equations with irregular coefficients have been studied from various perspectives: Mean field SDEs with a discontinuous drift appear for example in \cite{bauer_irregular_drift, leobacher_discont_drift}, where the discontinuity occurs in the space component, and in \cite{huang_discont_wasserstein} where the drift coefficient is discontinuous in the measure component with respect to the Wasserstein distance. For ordinary SDEs --- that is, equations whose coefficients do not depend on the law of the unknown process --- discontinuous drift has been studied for example in \cite{disc_drift_numeric_2, disc_drift_numeric_1} \newer{from the viewpoint of} numerics.

Mean field SDEs with a diffusion coefficient with relaxed regularity, on the other hand, have been studied in \cite{stannat_mehri_lyapunov, mishura_true_mckean_vlasov}, but the major difference to our setting is that \newer{it is assumed} that the coefficient functions are of the type
\begin{equation} \label{eq_true_mckean_vlasov_struct}
\sigma(t, x, \mu) = \int_{\R^d} \bar{\sigma}(t, x, y) \od \mu (y), \quad b(t, x, \mu) = \int_{\R^d} \bar{b}(t, x, y) \od \mu (y)
\end{equation}
for some functions $\bar{\sigma}$ and $\bar{b}$. In general, \newer{there does not necessarily exist a function $\bar \sigma$ such that \eqref{eq_true_mckean_vlasov_struct} holds for a} function \newer{$\sigma$} of the type \eqref{eq_type_of_disc}. Assuming the structure \eqref{eq_true_mckean_vlasov_struct} it is shown in \cite{zhang_discrete} that under sufficient conditions the coefficient functions can be allowed to be discontinuous in the measure component with respect to the topology of weak convergence. Without the mean field interaction, discontinuity in the diffusion coefficient has been studied for example in \cite{lejay_discontinuous_media}, and with a fractional Brownian motion as the driving process in \cite{garzon_discontinuous_fractional, torres_viitasaari}.

\subsection*{Structure of the article}

This article is organized as follows: In \cref{section_general_setting} we introduce the general setting and the \corr{standing} assumptions \corr{on the coefficient functions}. In \cref{sec_transformation} we show that these assumptions allow us to transform the equation \eqref{eq_very_general_equation} into a new SDE with deterministic coefficient functions. We give examples of equations for which a transformation exists. In \cref{section_existence} we prove a theorem \corr{on} the existence and uniqueness of a strong solution, and consider cases when the solution may only exist for a finite time due to an explosion in the variance. The proof of the existence result is based on a condition that ensures the strict monotonicity of the function $\varphi(\cdot) = \P(X_{\cdot} \leq r(\cdot))$, and violating this condition can result \corr{in} equations that have multiple solutions or no (global) solution, which is the topic of \cref{sec_counter_examples}.

\subsection*{Notation \corr{and conventions}}

We use the following notation and conventions:
\begin{itemize}
    \item $\N = \left\{1, 2, ... \right\}$.
    \item $\inf \emptyset = \infty$ and $\sup \emptyset = -\infty$.
    \item Given a topological space $X$ we denote by $\cB(X)$ the Borel $\sigma$-algebra in $X$, that is, the smallest $\sigma$-algebra containing all open sets in $X$.
    \item $\Phi$ denotes the cumulative \new{distribution} function of the standard normal distribution $\cN(0, 1)$.
    \item $\xi \sim \lognormal(\mu, \ivar^2)$ if and only if $\log(\xi) \sim \cN(\imean, \ivar^2)$.
    \item \corr{Given sets $A, B \in \cB(\R)$ we say that a function $f : A \to B$ is measurable if it is Borel measurable, that is, $(\cB(A), \cB(B))$-measurable.}
    \item \corr{The space $\cC^1(\rplus; \rplus)$ contains all \newer{continuous} functions $f : \rplus \to \rplus$ that are differentiable on $(0, \infty)$ and the derivative can be continuously extended to $0$. }
\end{itemize}


\section{General setting} \label{section_general_setting}

Throughout this article we suppose that $(\Omega, \cF, \P, \F := \left( \cF_t \right)_{t \in \rplus})$ is a stochastic basis satisfying the usual conditions, and we let $B = (B_t)_{t \in \rplus}$ be a \newer{$1$-}dimensional $\F$--Brownian motion. 

\changedagain{We assume} coefficient functions
\begin{align*}
\sigma & : \rplus \times \R \to \R, \\
b & : \rplus \times \R \to \R,
\end{align*}
\changedagain{where $\sigma$ is jointly continuous, and $b$ is continuous on the set $D \subseteq \rplus \times \R$ defined in \eqref{eq_def_D} and jointly measurable elsewhere ($b$ will be restricted to $D$ later). Given} an $\cF_0$-measurable initial condition $x_0$ and a measurable reference function $r : \rplus \to \R$ we study mean field stochastic differential equations of the form
\begin{equation} \label{eq_target_sde}
\begin{cases}
\begin{aligned}
X_t & = x_0 + \int_0^t \left( \sum_{n=1}^\infty \ind_{\left\{ \varphi(s) \in I_n \right\}} \alpha_n(s) \right) \sigma(s, X_s) \od B_s + \int_0^t b(s, X_s) \od s, \quad t \in \rplus, \\
\varphi(t) & = \P(X_t \leq r(t)),
\end{aligned}
\end{cases}
\end{equation}
where $(I_n)_{n=1}^\infty \subset \cB([0,1))$ forms a partition of $[0, 1)$ and $\alpha_n : \rplus \to \R$ is bounded and \newer{continuous} for all $n \in \N$.

\corr{
\begin{definition} \label{def_solution}
A strong solution to \eqref{eq_target_sde} is a pair $(X, T)$, where $T \in (0, \infty]$ is the lifetime of the solution and $X = (X)_{t \in [0, T)}$ is an adapted process with continuous sample paths such that 
\begin{enumerate}[label=\rm{(\arabic*)}]
    \item the equation \eqref{eq_target_sde} holds for all $t \in [0, T)$ $\P$-almost surely, and 
    \item for all $S < T$ one has 
    \begin{equation*}
    \int_0^S \left| b(s, X_s) \right| \od s + \int_0^S \left( \sum_{n=1}^\infty \ind_{\left\{ \varphi(s) \in I_n \right\}} \alpha_n(s)^2 \right) \left| \sigma(s, X_s) \right|^2 \od s < \infty
    \end{equation*}
    $\P$-almost surely, where $\varphi(t) = \P(X_t \leq r(t))$.
\end{enumerate}
A solution $(X, T)$ is unique (up to $T$) if for any other solution $(\widetilde{X}, S)$ one has $X_t = \widetilde{X}_t$ for all $t \in [0, T \wedge S)$ \newer{$\P$-almost surely}. Furthermore, a lifetime $\tmax \in (0, \infty]$ is maximal if there exists a unique solution $(X, \tmax)$ to \eqref{eq_target_sde} such that for any other solution $(\widetilde{X}, S)$ one has $S \leq \tmax$.
\end{definition}

\begin{remark}
If the lifetime of a solution is clear from the context, then we simply call the process $X$ itself a solution without explicitly specifying the lifetime.
\end{remark}
}

Using the notation
\begin{equation*}
U_t := \left\{ x \in \R \mid \sigma(t, x) > 0 \right\}, \quad t \in \rplus,
\end{equation*}
and 
\begin{equation} \label{eq_def_D}
D := \left\{ (t, x) \in \rplus \times \R \mid x \in U_t \right\} = \left\{ (t, x) \in \rplus \times \R \mid \sigma(t, x) > 0 \right\},
\end{equation}
our standing assumption for the functions $\sigma$ and $b$ states as follows:
\begin{assumption} \label{ass_initial}
There exist bounded functions $\sigma_1, \sigma_2 \in \cC^1(\rplus; \rplus)$ such that the following two conditions hold:
\begin{enumerate}[label=\bf{(As-$\sigma$)}]
    \item For all $(t, x) \in \rplus \times \R$ one has $\sigma(t, x) = \sigma_1(t)x + \sigma_2(t)$, and either 
    \begin{enumerate}[label=\rm{(\roman*)}]
        \item $\sigma_1 \equiv 0$ and $\sigma_2(t) > 0$ for all $t \in \rplus$, or
        \item $\sigma_1(t) > 0$ for all $t \in \rplus$.
    \end{enumerate} \label{standing_as_sigma}
\end{enumerate}
\begin{enumerate}[label=\bf{(As-$b$)}]
\setcounter{enumi}{1}
    \item \changedagain{The function $b : D \to \R$ is defined by 
    \begin{equation*}
    b(t, x) := \corr{\frac{\sigma_2'(t)}{\sigma_2(t)} }x + k(t),
    \end{equation*}
    if $\sigma_1 \equiv 0$, and by
    \begin{align*}
    b(t, x) := \, & \frac{\sigma_1'(t) \left[ \sigma_1(t) x + \sigma_2(t) \right] \log(\sigma_1(t)x + \sigma_2(t)) - \sigma_1(t)\sigma_2'(t) + \sigma_1'(t) \sigma_2(t)}{\sigma_1(t)^2} \\
    & + \ell(t) \left[ \sigma_1(t) x + \sigma_2(t) \right],
    \end{align*}
    if $\sigma_1(t) > 0$ for all $t \in \rplus$, where $k, \ell : \rplus \to \R$ are continuous functions. Accordingly, $b(t, \cdot)$ solves the ODE
    \begin{equation} \label{eq_ode_b}
    \sigma_1(t) b(t, x) = \partial_x b(t, x) \left[\sigma_1(t) x + \sigma_2(t) \right] - \sigma_1'(t) x - \sigma_2'(t), \quad x \in U_t,
    \end{equation}
    for all $t \in \rplus$.
    }
    \label{standing_as_b}
\end{enumerate}
\end{assumption}

One observes that under \ref{standing_as_sigma} one has 
\begin{equation*}
U_t = \left( -\frac{\sigma_2(t)}{\sigma_1(t)}, \infty \right),
\end{equation*}
where we use the convention \corr{$-\frac{x}{0} = -\infty$} for $x > 0$. \newer{In particular, if $\sigma_1 \equiv 0$, then $D = \rplus \times \R$.}

\corr{
\begin{definition}
\changedagain{We let} 
\begin{equation*}
D_+ := \left\{ (t, x) \in (0, \infty) \times \R \mid x \in U_t \right\} = \left\{ (t, x) \in (0, \infty) \times \R \mid \sigma(t, x) > 0 \right\},
\end{equation*}
which is an open set by \ref{standing_as_sigma}. We say that $F \in \cC^{1, 2}(D)$ if and only if the partial derivatives \changed{
\begin{equation*}
\frac{\partial F}{\partial x}, \,\, \frac{\partial^2 F}{\partial x^2}, \,\, \frac{\partial F}{\partial t}, \,\, \frac{\partial^2 F}{\partial x \partial t}, \,\, \frac{\partial^2 F}{\partial t \partial x}
\end{equation*}
\changedagain{exist and} are continuous on $D_+$, and they can be continuously extended to the set $\left\{ 0 \right\} \times U_0$.
}

\end{definition}
}

In \cref{sec_transformation} we show that under \ref{standing_as_sigma} and \ref{standing_as_b} we can transform the equation \eqref{eq_target_sde} into an SDE with deterministic coefficients, that is, there exists a transformation $F \in \cC^{1, 2}(D)$ invertible in the second component such that for $\xi_0 = F(0, x_0)$ and $\rho(t) = F(t, r(t))$ the process $Y := F(\cdot, X_{\cdot})$ solves 
\begin{equation} \label{eq_gaussian_sde}
\begin{cases}
\begin{aligned}
Y_t & = \xi_0 + \int_0^t \sum_{n=1}^\infty \ind_{\left\{ \varphi(s) \in I_n \right\}} \left[ \alpha_n(s) \od B_s + \beta_n(s) \od s \right], \quad t \in \rplus, \\
\varphi(t) & = \P(Y_t \leq \rho(t)),
\end{aligned}
\end{cases}
\end{equation}
where 
\begin{equation} \label{eq_beta}
\beta_n(t) := 
\begin{cases}
\begin{aligned}
& \frac{b(t, 1)}{\sigma_2(t)} - \frac{\sigma_2'(t)}{\newer{\sigma_2(t)^2}}, \quad && \sigma_1 \equiv 0, \\
& \frac{b(t, 1)}{\sigma_1(t) + \sigma_2(t)} - \frac{1}{2} \alpha_n(t)^2 \sigma_1(t) \\
& \hspace{2.0cm} + \frac{1}{\sigma_1(t)} \left[ \frac{\sigma_1'(t) + \sigma_2'(t)}{\sigma_1(t) + \sigma_2(t)} - \frac{\sigma_1'(t)}{\sigma_1(t)} \log( \sigma_1(t) + \sigma_2(t) ) \right], \quad && \sigma_1 > 0,
\end{aligned}
\end{cases}
\end{equation}
for $n \in \N$ and $t \in \rplus$. \corr{Here the notation $\sigma_1 > 0$ means that $\sigma_1(t) > 0$ for all $t \in \rplus$. \newer{Since the functions $(\alpha_n)_{n=1}^\infty$ and $b(\cdot, 1)$ are continuous (the continuity of $b(\cdot, 1)$ \changed{follows from \ref{standing_as_b} since $1 \in U_t$ for all $t \in \rplus$}), the function $\beta_n$ is also continuous for all $n \in \N$.} A solution to \eqref{eq_gaussian_sde} is defined similarly as for the original equation:

\begin{definition} \label{def_solution_Y}
A strong solution to \eqref{eq_gaussian_sde} is a pair $(Y, T)$, where $T \in (0, \infty]$ is the lifetime of the solution and $Y = (Y)_{t \in [0, T)}$ is an adapted process with continuous sample paths such that the equation \eqref{eq_gaussian_sde} holds for all $t \in [0, T)$ $\P$-almost surely, and for all $S < T$ one has
\begin{equation*}
\int_0^S \sum_{n=1}^\infty \ind_{\left\{ \varphi(s) \in I_n \right\}} \left( \alpha_n(s)^2 + \left| \beta_n \right| \right) \od s < \infty
\end{equation*}
$\P$-almost surely, where $\varphi(t) = \P(Y_t \leq \rho(t))$. The uniqueness of a solution and the maximality of a lifetime are defined in the same way as in \cref{def_solution}.
\end{definition}

The regularity assumptions on the transformation $F$ are the main reason why the coefficient functions in equation \eqref{eq_target_sde}  have a more specific form than in \eqref{eq_very_general_equation}: if each $\sigma_n$ in \eqref{eq_very_general_equation} had a different structure, then \newer{more conditions would be required to guarantee} the time regularity in $F$, hence it is convenient to assume that only the time-dependent factor in front of the diffusion coefficient changes. Furthermore, if $\sigma$ and $b$ in \eqref{eq_target_sde} depended on a measure component, then one would require a transformation that also depends on a measure component.
}


\section{Transformation into an SDE with deterministic coefficients} \label{sec_transformation}
In this section we show that equation \eqref{eq_target_sde} can be transformed into the SDE \eqref{eq_gaussian_sde} and back. We propose the transformation $F \in \cC^{1, 2}(D)$ defined by
\begin{equation} \label{eq_def_F_initial}
    F(t, x) = \begin{cases}
    \begin{aligned}
    & \frac{x}{\sigma_2(t)}, \quad && \sigma_1 \equiv 0, \\
    & \frac{\log( \sigma_1(t) x + \sigma_2(t) )}{\sigma_1(t)}, \quad && \sigma_1 > 0.
    \end{aligned}
    \end{cases}
\end{equation}
\newer{One observes that the function $F(t, \cdot)$ is onto $\R$, that is, $F(t, \newest{U_t}) = \R$ for all $t \in \rplus$.} Since $F(t, \cdot) : U_t \to \R$ is monotone, it has an inverse $G(t, \cdot) := F(t, \cdot)^{-1} : \R \to U_t$ given by
\begin{equation} \label{eq_def_G_initial}
    G(t, \newer{y}) = \begin{cases}
    \begin{aligned}
    & \sigma_2(t) \newer{y}, \quad && \sigma_1 \equiv 0, \\
    & \frac{\exp( \sigma_1(t) \newer{y} ) - \sigma_2(t)}{\sigma_1(t)}, \quad && \sigma_1 > 0.
    \end{aligned}
    \end{cases}
\end{equation}
One sees that $G \in \cC^{1, 2}(\rplus \times \R; D)$.

\begin{proposition} \label{prop_only_transform}
Suppose that \cref{ass_initial} holds. \corr{Then the following is true:}
\begin{enumerate}[label=\rm{(\roman*)}]
    \item \corr{Assume that $x_0$ is an $\cF_0$-measurable random variable such that \newer{$x_0(\omega) \in U_0$ for all $\omega \in \Omega$}, and the function $r : \rplus \to \R$ is measurable and satisfies $r(t) \in U_t$ for all $t \in \rplus$. Then} for any \corr{strong} solution \corr{$(X, T)$} to \eqref{eq_target_sde} such that \newer{$X_t(\omega) \in U_t$ for all $\omega \in \Omega$ and} $t \in [0, T)$, the process $Y_t := F(t, X_t)$, where $F$ is defined in \eqref{eq_def_F_initial}, solves \eqref{eq_gaussian_sde} on $[0, T)$ for $\xi_0 = F(0, x_0)$ and $\rho(t) = F(t, r(t))$. 
    \label{transf_eq_item_1}
    \item \corr{Assume that $\xi_0$ is an $\cF_0$-measurable random variable and the function $\rho : \rplus \to \R$ is measurable. Then} for any \corr{strong} solution \corr{$(Y, T)$} to \eqref{eq_gaussian_sde} the process $X_t := G(t, Y_t)$, where $G$ is defined in \eqref{eq_def_G_initial}, satisfies \newer{$X_t(\omega) \in U_t$ for all $(\omega, t) \in \Omega \times [0, T)$} and solves \eqref{eq_target_sde} on $[0, T)$ for $x_0 = G(0, \xi_0)$ and $r(t) = G(t, \rho(t))$. \label{transf_eq_item_2}
\end{enumerate}
\end{proposition}

\begin{proof}$ $\newline

{\noindent
\fbox{Proof of \ref{transf_eq_item_1}:}
Let $X = (X_t)_{t \in [0, T)}$ be a solution to \eqref{eq_target_sde} such that \newer{$X_t(\omega) \in U_t$ for all $(\omega, t) \in \Omega \times [0, T)$}. We apply Itô's formula to the function $F$ and the process $X$ to get 
\begin{align} \label{eq_ito_to_X}
\corr{Y_t} = F(t, X_t) = F(0, x_0) & + \int_0^t \left( \sum_{n=1}^\infty \ind_{\left\{ \varphi(s) \in I_n \right\}} \alpha_n(s) \right) \partial_x F(s, X_s) \sigma(s, X_s) \od B_s \nonumber \\
& + \frac{1}{2} \int_0^t \left( \sum_{n=1}^\infty \ind_{\left\{ \varphi(s) \in I_n \right\}} \alpha_n(s)^2 \right) \partial_{xx} F(s, X_s) \sigma(s, X_s)^2 \od s \nonumber \\
& + \int_0^t \left[ \partial_x F(s, X_s) b(s, X_s) + \partial_s F(s, X_s) \right] \od s 
\end{align}
for all $t \in [0, T)$. We want to show that
\begin{align}
1 & = \partial_x F(t, x) \sigma(t, x), \label{eq_alpha_1} \\ 
\beta_n(t) & = \frac{1}{2} \alpha_n(t)^2  \partial_{xx} F(t, x) \sigma(t, x)^2 + \partial_x F(t, x) b(t, x) + \partial_t F(t, x) \label{eq_beta_1}
\end{align}
for all $(t, x) \in \newer{D_+}$ and $n \in \N$. 
\corr{The property \eqref{eq_alpha_1} follows immediately by differentiating $F$ with respect to $x$. To prove \eqref{eq_beta_1} we use \eqref{eq_alpha_1} and the definition of $F$ in \eqref{eq_def_F_initial} to compute the partial derivatives
\begin{equation} \label{eq_partial_F}
\partial_x F(t, x) = \frac{1}{\sigma(t, x)}, \quad \partial_{xx} F(t, x) = -\frac{\partial_{x} \sigma(t, x)}{\sigma(t, x)^2}, \quad \partial_x ( \partial_t F(t, x) ) = -\frac{\partial_t \sigma(t, x)}{\sigma(t, x)^2}
\end{equation}
for $(t, x) \in \newer{D_+}$. Using \eqref{eq_partial_F} we see that 
\begin{equation*}
\partial_{xx} F(t, x) \sigma(t, x)^2 = -\partial_x \sigma(t, x),
\end{equation*}
hence the right-hand side of \eqref{eq_beta_1} is differentiable with respect to $x$: Recalling that $\sigma(t, x) = \sigma_1(t) x + \sigma_2(t)$ by \ref{standing_as_sigma} we use \eqref{eq_partial_F} to compute} 
\begin{align*}
& \frac{\partial}{\partial x} \left[ \frac{1}{2} \alpha_n(t)^2 \partial_{xx} F(t, x) \sigma(t, x)^2 + \partial_x F(t, x) b(t, x) + \partial_t F(t, x) \right] \\
&\corr{= -\frac{1}{2} \alpha_n(t)^2 \partial_{xx} \sigma(t, x) + \frac{\partial}{\partial x} \left[ \frac{b(t, x)}{\sigma(t, x)}\right] -\frac{\partial_t \sigma(t, x)}{\sigma(t, x)^2} } \\
& = \frac{\left[\sigma_1(t) x + \sigma_2(t) \right] \partial_x b(t, x) - \sigma_1(t) b(t, x) - \sigma_1'(t) x - \sigma_2'(t) }{\sigma(t, x)^2} \\
& = 0
\end{align*}
for all $(t, x) \in \newer{D_+}$, \corr{where the last equality follows from} \ref{standing_as_b}. Since $1 \in U_t$ for all $t \in \rplus$ under the assumption \ref{standing_as_sigma}, we can let $x = 1$ in \eqref{eq_beta_1} to get \eqref{eq_beta}.}

We conclude the proof of the first part by observing that since $F(t, \cdot)$ is invertible and $r(t) \in U_t$ for all $t \in [0, T)$, one has 
\begin{equation*}
\P(X_t \leq r(t)) = \P(F(t, X_t) \leq F(t, r(t))) = \P(Y \leq \rho(t))
\end{equation*}
for all $t \in [0, T)$.

\bigskip

{\noindent
\fbox{Proof of \ref{transf_eq_item_2}:} 
Let $Y = (Y_t)_{t \in [0, T)}$ be a solution to \eqref{eq_gaussian_sde} on some interval $[0, T) \subseteq \rplus$. \newest{Now one has} $X_t = G(t, Y_t) \in U_t$ for all $t \in [0, T)$.}

We proceed in the same way as in the first case: by recalling that the inverse of $G(t, \cdot)$ is $F(t, \cdot)$ we apply Itô's formula to get 
\begin{align*}
X_t = G(0, \xi_0) & + \int_0^t \left( \sum_{n=1}^\infty \ind_{\left\{ \varphi(s) \in I_n \right\}} \alpha_n(s) \right) \new{\newer{\partial_y G}}(s, F(s, X_s)) \od B_s \\
& + \int_0^t \sum_{n=1}^\infty \ind_{\left\{ \varphi(s) \in I_n \right\}} \left[ \frac{1}{2} \alpha_n(s)^2 \new{\newer{\partial_{yy} G}}(s, F(s, X_s)) + \beta_n(s) \new{\newer{\partial_y G}}(s, F(s, X_s)) \right] \od s \\
& + \int_0^t \new{[\partial_s G]}(s, F(s, X_s)) \od s
\end{align*}
for $t \in [0, T)$.
We want to show that 
\begin{align}
\sigma(t, x) & = \new{\newer{\partial_y G}}(t, F(t, x)) , \label{eq_alpha_2} \\
b(t, x) & = \frac{1}{2} \alpha_n(t)^2 \new{\newer{\partial_{yy} G}}(t, F(t, x)) + \beta_n(t) \new{\newer{\partial_y G}}(t, F(t, x)) + \new{[\partial_t G]}(t, F(t, x)) \label{eq_beta_2}
\end{align} 
for $(t, x) \in \newer{D_+}$ and $n \in \N$ when using $\beta_n$ defined in \eqref{eq_beta}. 
The property \eqref{eq_alpha_2} is straightforward to show by computing the partial derivative $\newer{\partial_{y} G}$ \new{using \eqref{eq_def_G_initial}}. To prove \eqref{eq_beta_2}, we \new{first observe that by \eqref{eq_alpha_1} and \eqref{eq_alpha_2} it holds
\begin{align*}
\newer{\partial_{y} G} (t, y) = \frac{1}{\partial_x F (t, G(t, y))}
\end{align*}
for $(t, y) \in \rplus \times \R$, where we used that $F(t, \cdot)^{-1} = G(t, \cdot)$. This in turn implies 
\begin{align} \label{eq_partial_G1}
\newer{\partial_{yy} G}(t, F(t, x)) = - \left. \frac{\newer{\partial_{y} G}(t, y) \partial_{xx} F (t, G(t, y)) }{\partial_x F(t, G(t, y))^2}  \right|_{y = F(t, x)} = \left[ -\left( \partial_{xx} F(t, x) \right) \sigma(t, x)^2 \right] \sigma(t, x)
\end{align}
for $(t, x) \in \newer{D_+}$ by \eqref{eq_partial_F}. Second, we use \eqref{eq_def_F_initial} and \eqref{eq_def_G_initial} to verify that 
\begin{equation}  \label{eq_partial_G2}
\left[ \partial_t G \right] (t, F(t, x)) = -\left[ \partial_t F(t, x) \right] \sigma(t, x)
\end{equation}
for all $(t, x) \in \newer{D_+}$.} \corr{We exploit \eqref{eq_alpha_2}, \eqref{eq_partial_G1} \new{and \eqref{eq_partial_G2}} to get 
\begin{align*}
& \frac{1}{2} \alpha_n(t)^2 \newer{\partial_{yy} G}(t, F(t, x)) + \beta_n(t)\newer{\partial_{y} G}(t, F(t, x)) + \newer{[\partial_t G]}(t, F(t, x)) \\
& = \sigma(t, x) \left[ -\frac{1}{2} \alpha_n(t)^2 \partial_{xx} F(t, x) \sigma(t, x)^2 - \partial_t F(t, x) + \beta_n(t) \right] \\
& = \sigma(t, x) \left[ \partial_x F(t, x) b(t, x) \right] \\
& = \sigma(t, x) \frac{b(t, x)}{\sigma(t, x)} \\
& = b(t, x)
\end{align*}
for all $(t, x) \in \newer{D_+}$, where we replaced $\beta_n$ \newest{by} the right-hand side of \eqref{eq_beta_1} to obtain the second equality.
}

Moreover, since $G(t, \cdot)$ is invertible, we complete the proof by observing that 
\begin{equation*}
\P(Y_t \leq \rho(t)) = \P(G(t, Y_t) \leq G(t, \rho(t))) = \P(X \leq r(t))
\end{equation*}
for $t \in [0, T)$.
\end{proof}

\begin{remark}
If we have $\alpha_n \equiv 1$ for all $n \in \N$, then one sees that the transformation $F$ turns the diffusion \corr{coefficient} into a \corr{constant function}. This is known as the Lamperti transformation \cite[Chapter 7.1]{sarkka}.
\end{remark}

Next we provide three simple examples of equations that can be transformed into an SDE with deterministic coefficients. \newer{In \cref{section_existence} we require that the transformed initial condition is normally distributed, that is, $\xi_0 = F(0, x_0) \sim \cN(\imean, \ivar^2)$ for some parameters $\imean \in \R$ and $\ivar^2 > 0$. This is also taken into account in the examples below.}

\begin{example}[Gaussian SDE]
Given a positive function $c \in \cC^1(\rplus)$ \newer{such that $c(0) = 1$}, and any \newer{continuously differentiable} function $\newest{a} : \rplus \to \R$ consider
\begin{equation*}
X_t = x_0 + \int_0^t \left( \sum_{n=1}^\infty \ind_{\left\{ \P(X_s \leq r(s)) \in I_n \right\}} \alpha_n(s) \right) c(s) \od B_s + \int_0^t \left[ \frac{c'(s)}{\newer{c(s)}} X_s + \frac{a(s)}{\newer{c(s)}} \right] \od s,
\end{equation*} 
where $x_0 \sim \cN(\imean, \ivar^2)$. \newest{Here $\sigma_1 \equiv 0$, $\sigma_2(t) = c(t)$ and $b(t, x) = \frac{c'(t)}{c(t)}x + \frac{a(t)}{c(t)}$.} The transformation is $F(t, x) = \frac{x}{c(t)}$, and the transformed drift coefficient functions are
\begin{align*}
\beta_n(t) = \frac{a(t)}{\newer{c(t)^2}}.
\end{align*}
In particular the functions $\beta_n$ are independent of $n$. One sees that if $c \equiv 1$, then $F$ is \newest{the} identity map since the equation is already in the desired form, and if $a \equiv 0$, then $\beta_n \equiv 0$.
\end{example}

\begin{example}[Linear SDE] \label{ex_linear_SDE}
Given a \newer{continuous and bounded} function $\changed{b} : \rplus \to \R$ consider
\begin{equation*}
X_t = x_0 + \int_0^t \left( \sum_{n=1}^\infty \ind_{\left\{ \P(X_s \leq r(s)) \in I_n \right\}} \alpha_n(s) \right)  X_s \od B_s + \int_0^t b(s) X_s \od s, \quad t \in \rplus,
\end{equation*}
where $x_0 \sim \lognormal(\imean, \ivar^2)$. \newest{Here $\sigma_1 \equiv 1$, $\sigma_2 \equiv 0$ and $b(t, x) = b(t)x$.} The transformation is $F(t, x) = \log(x)$ and the transformed drift coefficient functions are 
\begin{align*}
& \beta_n(t) = b(t) - \frac{1}{2} \alpha_n(t)^2.
\end{align*}
\newer{By the definition of the log-normal distribution one has $F(0, x_0) = \log(x_0) \sim \cN(\imean, \ivar^2)$}
\end{example}

\begin{example}[SDE with logarithm in the drift] \label{ex_loglinear_SDE}
Given a positive function $c \in \cC^1(\rplus)$ \newer{such that $c(0) = 1$} consider
\begin{equation} \label{eq_sde_with_log}
X_t = x_0 + \int_0^t \left( \sum_{n=1}^\infty \ind_{\left\{ \P(X_s \leq r(s)) \in I_n \right\}} \alpha_n(s) \right)  c(s) X_s \od B_s + \int_0^t \frac{c'(s)}{c(s)} X_s \log(c(s) X_s) \od s, \quad t \in \rplus,
\end{equation}
where $x_0 \sim \lognormal(\imean, \ivar^2)$. \newest{Here $\sigma_1(t) = c(t)$, $\sigma_2 \equiv 0$ and $b(t, x) = \frac{c'(t)}{c(t)} x \log(c(t) x)$.} The transformation is $F(t, x) = \frac{\log(c(t) x)}{c(t)} $ and the transformed drift coefficient functions are 
\begin{align*}
& \beta_n(t) = \frac{c'(t)}{c(t)^2} - \frac{1}{2} \alpha_n(t)^2 c(t).
\end{align*}
\newer{Since $c(0) = 1$ one has $F(0, x_0) = \log(x_0) \sim \cN(\imean, \ivar^2$).}

As a special case of \eqref{eq_sde_with_log} one obtains, by choosing $c(t) = \exp(-f(t))$ for any \newer{non-negative} function $f \in \cC^1(\rplus)$ \newer{such that $f(0) = 0$}, the SDE 
\begin{align*}
X_t = x_0 & + \int_0^t \left( \sum_{n=1}^\infty \ind_{\left\{ \P(X_s \leq r(s)) \in I_n \right\}} \alpha_n(s) \right) \frac{X_s}{\exp(f(s))} \od B_s \\
& + \int_0^t f'(s) \left[f(s)X_s - X_s \log(X_s)  \right] \od s, \quad t \in \rplus.
\end{align*}

\end{example}

\changed{
\begin{remark}
\changedagain{
The examples above cover basic types of SDEs for which \cref{prop_only_transform} can be applied. At the same time, they are important for the following reasons:
}
\begin{enumerate}[label=\rm{(\arabic*)}]
    \item They give concrete examples of coefficient functions that satisfy \cref{ass_initial}.
    \item Using these examples one can find $\alpha_n$ and $\beta_n$ that satisfy \ref{hypo_increasing} below. This is an essential assumption in the existence and uniqueness theorem, \cref{theo_existence}, which is proven in \cref{section_existence}.
    \item In \cref{sec_counter_examples}, where we study non-uniqueness and non-existence of a solution, the results are formulated only for the transformed equation \cref{eq_gaussian_sde}. However, with the help of \cref{ex_linear_SDE} and \cref{ex_loglinear_SDE} one can find examples of functions $\alpha_1, \alpha_2, \beta_1$ and $\beta_2$ that satisfy the assumptions of \cref{prop_many_solutions}, \cref{theo_no_solution} and \cref{rem_inf_solutions}.
\end{enumerate}
\end{remark}
}


\section{Existence and uniqueness of a strong solution} \label{section_existence}

In this section we show that under additional assumptions the equation \eqref{eq_target_sde} has a unique strong solution defined on some maximal interval, which means that the solution cannot be continuously extended beyond this interval. 

Our standing assumption for this section is the following:
\begin{assumption} \label{ass_existence}
Suppose that in addition to \cref{ass_initial} the following also holds:
\begin{enumerate}[label=\bf{(As-$x_0$)}]
    \item The initial condition $x_0$ is $\cF_0$-measurable, satisfies $\P(x_0 \in U_0) = 1$ and $F(0, x_0) \sim \cN(\imean, \ivar^2)$, where $F$ is the transformation defined in \eqref{eq_def_F_initial}, \newer{and $\imean \in \R$ and $\ivar^2 > 0$ satisfy
    \begin{equation} \label{eq_cond_mu_initial}
    \left| \imean \right| \leq \frac{1}{2} \ivar^2.
    \end{equation}}
    \label{hypo_x0}
\end{enumerate}
\begin{enumerate}[label=\bf{(As-$I$)}]
    \item There exists $0 =: y_0 < y_1 < y_2 < ... \uparrow 1$ such that $I_n = [y_{n-1}, y_n)$. \label{hypo_almost_intervals} 
\end{enumerate}
\begin{enumerate}[label=\bf{(As-$r$)}]
    \item $r(t) \in U_t$ for all $t \in \rplus$, the function \newest{$r$} is continuous\newer{ly differentiable}, and \newest{the function $\rho(t) := F(t, r(t))$ is} non-decreasing and \newest{satisfies}  $\sup_{t \in \rplus} \rho(t) \leq \imean$. \label{hypo_r}
\end{enumerate}
\begin{enumerate}[label=\bf{(As-$\beta_n$)}]
    \item For all $n \in \N$ and $t \in \rplus$ it holds
    \begin{equation} \label{eq_cond_strong_diffusion}
    -\frac{1}{2} \alpha_n(t)^2 \leq \beta_n(t) \leq -\frac{1}{4} \alpha_n(t)^2,
    \end{equation} 
    where $\beta_n$ is defined in \eqref{eq_beta}, and \newer{for all $n \in \N$} at least one of the inequalities \newer{in \eqref{eq_cond_strong_diffusion}} is strict \newer{uniformly for all $t \in \rplus$}. \label{hypo_increasing}
\end{enumerate}
\end{assumption}

The main result of this section states as follows:
\begin{theorem} \label{theo_existence}
Supposing that \cref{ass_existence} holds the equation \eqref{eq_target_sde} has a strong solution \corr{$(X, \tmax)$ for some $\tmax \in (0, \infty]$} such that $X_t \in U_t$ for all $t \in [0, \tmax)$. Moreover:
\begin{enumerate} [label=\rm{(\alph*)}]
    \item \textbf{Uniqueness}: The solution is unique in the sense that for any other solution $\widetilde{X} = (\widetilde{X}_t)_{t \in [0, S)}$ to \eqref{eq_target_sde} defined on some interval $[0, S) \subseteq \rplus$ such that $\widetilde{X}_t \in U_t$ for all $t \in [0, S)$, one has $S \leq \tmax$ and $X_t = \widetilde{X}_t$ $\P$-almost surely for all $t \in [0, \tmax \wedge S) = [0, S)$.
    \item \textbf{Maximality}: The \corr{lifetime} $\tmax \in (0, \infty]$ is maximal: \newest{there is no random variable $Y_{\tmax}$ such that $F(t, X_t) \to Y_{\tmax}$ in distribution as $t \to \tmax^-$.} Consequently, the solution \newest{$X$} cannot be extended beyond $\tmax$.
    \item \textbf{Globality}: If there is a constant $C > 0$ such that 
    \begin{equation} \label{eq_as_bounded_beta}
    \sup_{n \in \N} \left( \sup_{t \geq 0} \alpha_{n}(t)^2 \right) \leq C,
    \end{equation}
    then $\tmax = \infty$.
    \item \textbf{Bijectivity} (of $\varphi$): If also 
    \begin{equation} \label{eq_positive_alphas}
    \inf_{n \in \N} \left( \inf_{t \in \rplus} \alpha_n(t)^2 \right) > 0,
    \end{equation}
    then the map $\varphi(\cdot) - \varphi(0) : [0, \tmax) \to [0, 1)$ is \newer{continuous,} strictly increasing and for all $n \in \N$ it holds
    \begin{equation} \label{eq_finite_levels}
    \inf \left\{ t \in \rplus \mid \varphi(t) \geq y_n \right\} < \infty,
    \end{equation}
    that is, every level $y_n$ above $\varphi(0)$ is reached in a finite time.
\end{enumerate}
\end{theorem}

To prove \cref{theo_existence} we require an auxiliary lemma \corr{that enables us to} determine whether the function $\varphi$ is strictly increasing or decreasing on a given interval.

\begin{lemma} \label{lemma_monotonicity} 
Suppose that $0 \leq t_0 < t_1 \leq \infty$, $\rho : [t_0, t_1) \to \R$ is a \newer{continuously differentiable} function and $\xi_{t_0} \sim \cN(\imean + \bar \mu_{t_0}, \ivar^2 + \bar \sigma_{t_0}^2)$ is an $\cF_{t_0}$-measurable random variable, \corr{where $\imean, \bar \mu_{t_0} \in \R$, $\ivar^2 > 0$ and $\bar \sigma_{t_0}^2 \geq 0$} such that 
\begin{equation} \label{eq_cond_mu}
\left| \bar \mu_{t_0} \right| \leq \frac{1}{2} \bar \sigma_{t_0}^2.
\end{equation}
For a fixed $n_0 \in \N$ let
\begin{equation} \label{eq_Y_psi}
Y^{(n_0, t_0, \xi_{t_0})}_t := \xi_{t_0} + \int_{t_0}^t \left[\alpha_{n_0}(s) \od B_s + \beta_{n_0}(s)  \od s \right], \quad t \in [t_0, t_1),
\end{equation}
for some bounded and \newer{continuous} functions $\alpha_{n_0}, \beta_{n_0} : \rplus \to \R$. Then \corr{the function
\begin{equation*}
\psi(t) := \P(Y_t^{(n_0, t_0, \xi_{t_0})} \leq \rho(t)), \quad t \in [t_0, t_1),
\end{equation*}
}
\newer{is continuous} and the following holds:
\begin{enumerate} [label=\rm{(\roman*)}]
\item If $\rho$ is non-decreasing, $\sup_{t \in [t_0, t_1)} \rho(t) \leq \imean$ and
\begin{equation} \label{eq_beta_neg} 
-\frac{1}{2} \alpha_{n_0}(t)^2 \leq \beta_{n_0}(t) \leq -\frac{1}{4} \alpha_{n_0}(t)^2
\end{equation}
for all $t \in [t_0, t_1)$, then $\psi$ is non-decreasing on $[t_0, t_1)$. Furthermore, if either of the inequalities in \eqref{eq_beta_neg} is strict \newer{uniformly for all $t \in [t_0, t_1)$}, then $\psi$ is strictly increasing, and if \corr{additionally} $t_1 = \infty$ and
\begin{equation} \label{eq_alpha_goes_infinity}
\inf_{t \in [t_0, \infty)} \alpha_{n_0}(t)^2 > 0,
\end{equation}
then $\psi(t) \uparrow 1$ as $t \to \infty$.

\label{item_increasing}
\item If $\rho$ is non-increasing, $\inf_{t \in [t_0, t_1)} \rho(t) \geq \imean$ and
\begin{equation} \label{eq_beta_pos} 
 \frac{1}{4} \alpha_{n_0}(t)^2 \leq \beta_{n_0}(t) \leq \frac{1}{2} \alpha_{n_0}(t)^2
\end{equation}
for $t \in [t_0, t_1)$, then $\psi$ is non-increasing on $[t_0, t_1)$. Furthermore, if either of the inequalities in \eqref{eq_beta_pos} is strict \newer{uniformly for all $t \in [t_0, t_1)$}, then $\psi$ is strictly decreasing, and if \corr{additionally} $t_1 = \infty$ and \eqref{eq_alpha_goes_infinity} holds, then $\psi(t) \downarrow 0$ as $t \to \infty$.
\label{item_decreasing}
\end{enumerate}
\end{lemma}

\begin{proof}
We let
\begin{equation} \label{eq_func_f}
f(t) := \frac{ (\rho(t) - \imean) - \bar \mu_{t_0} - \int_{t_0}^{t} \beta_{n_0}(s) \od s  }{\sqrt{ \ivar^2 + \bar \sigma_{t_0}^2 + \int_{t_0}^{t} \alpha_{n_0}(s)^2 \od s }}
\end{equation}
for $t \in [t_0, t_1)$. \newer{Since $Y_t^{(n_0, t_0, \xi_{t_0})}$ is a sum of two independent normally distributed random variables, one has
\begin{equation*}
Y_t^{(n_0, t_0, \xi_{t_0})} \sim \cN \left(  \imean + \bar \mu_{t_0} + \int_{t_0}^{t} \beta_{n_0}(s) \od s, \ivar^2 + \bar \sigma_{t_0}^2 + \int_{t_0}^{t} \alpha_{n_0}(s)^2 \od s\right)
\end{equation*}
for all $t \in [t_0, t_1)$. Consequently}
\newer{
\begin{equation} \label{eq_expr_psi}
\psi(t) = \P(Y^{(n_0, t_0, \xi_{t_0})}_t \leq \rho(t)) = \Phi \left( \frac{\rho(t) - \E Y^{(n_0, t_0, \xi_{t_0})}_t}{\sqrt{\variance\left(Y^{(n_0, t_0, \xi_{t_0})}_t\right)}} \right) = \Phi(f(t))
\end{equation}
for $t \in [t_0, t_1)$,
}
where $\Phi$ is the cumulative \new{distribution} function of the standard normal distribution. \newer{The continuity of the function $\psi$ follows from expressions \eqref{eq_func_f} and \eqref{eq_expr_psi} and the continuity of $\Phi$.}


\newer{Since the functions $\alpha_{n_0}$ and $\beta_{n_0}$ are continuous and $\rho \in \cC^1(\rplus)$, the function $f$ is differentiable everywhere.} We compute
\begin{align*}
f'(t) = \frac{g(t)}{2 \left( \ivar^2 + \bar \sigma_{t_0}^2 + \int_{t_0}^{t} \alpha_{n_0} (s)^2 \od s \right)^{\frac{3}{2}}}
\end{align*}
for \newer{$t \in [t_0, t_1)$}, where
\begin{align*}
g(t)  = & \, \alpha_{n_0}(t)^2 \left[ \int_{t_0}^{t} \beta_{n_0}(s) \od s + \bar \mu_{t_0} + \left( \imean - \rho(t) \right) \right]  \\
& + 2 \left[\rho'(t) - \beta_{n_0}(t) \right] \left[ \ivar^2 + \bar \sigma_{t_0}^2 + \int_{t_0}^{t} \alpha_{n_0}(s)^2 \od s \right].
\end{align*}
Since $\Phi$ is strictly increasing it suffices to study the sign of the function $g$ \newer{on $(t_0, t_1)$.}
\bigskip

{\noindent
\fbox{Proof of \ref{item_increasing}:} Since now $\rho'(t) \geq 0$ and $\rho(t) \leq \imean$, we have by \eqref{eq_beta_neg} and \eqref{eq_cond_mu} that 
\begin{align}
g(t) & \geq -\alpha_{n_0}(t)^2 \left[ \frac{1}{2} \int_{t_0}^t \alpha_{n_0}(s)^2 \od s + \newer{\left| \bar \mu_{t_0} \right|} \right] + \frac{1}{2} \alpha_{n_0}(t)^2 \left[ \int_{t_0}^t \alpha_{n_0}(s)^2 \od s + \bar \sigma_{t_0}^2 \right] \label{eq_g1} \\
& = \alpha_{n_0}(t)^2 \left[ \frac{1}{2} \bar \sigma_{t_0}^2 - \newer{\left| \bar \mu_{t_0} \right|} \right] \geq 0. \nonumber
\end{align}
If either of the inequalities in \eqref{eq_beta_neg} is strict \newer{uniformly for all $t \in [t_0, t_1)$, then one has either $\int_{t_0}^{t} \beta_{n_0}(s) \od s > -\frac{1}{2} \int_{t_0}^t \alpha_{n_0}(s)^2 \od s$ or $-\beta_{n_0}(t) > \frac{1}{4} \alpha_{n_0}(t)^2$, which makes the inequality in \eqref{eq_g1} strict, implying that} $g(t) > 0$.}

To show the remaining claim let us suppose that \eqref{eq_alpha_goes_infinity} holds. \newer{Since $\rho$ is non-decreasing, it holds $\inf_{t \in \rplus} \rho(t) = \rho(0)$, and hence}
\begin{align*}
\psi(t) & \geq \Phi \left( \frac{ \newer{-\left| \rho(0) - \imean - \bar \mu_{t_0} \right|} - \int_{t_0}^t \beta_{n_0}(s) \od s }{\sqrt{ \ivar^2 + \bar \sigma_{t_0}^2 + \int_{t_0}^t \alpha_{n_0}(s)^2 \od s }} \right) \\
& \geq \Phi \left(\frac{\newer{-\left| \rho(0) - \imean - \bar \mu_{t_0} \right|}}{\ivar} + \frac{ \frac{1}{4} \left( \inf_{u \geq t_0} \alpha_{n_0}(u)^2 \right) (t - t_0) }{\sqrt{ \ivar^2 + \bar \sigma_{t_0}^2 + \left( \sup_{u \geq t_0} \alpha_{n_0}(u)^2 \right)(t - t_0) }} \right) \to 1
\end{align*}
as $t \to \infty$ by \eqref{eq_alpha_goes_infinity} and the boundedness of $\alpha_{n_0}$.

\bigskip

{\noindent
\fbox{Proof of \ref{item_decreasing}:} This time \corr{we have} $\rho'(t) \leq 0$ and $\rho(t) \geq \imean$, hence by \eqref{eq_beta_pos} and \eqref{eq_cond_mu} it holds 
\begin{align*}
g(t) & \leq \alpha_{n_0}(t)^2 \left[ \frac{1}{2} \int_0^t \alpha_{n_0}(s)^2 \od s + \bar \mu_{t_0} \right] - \frac{1}{2} \alpha_{n_0}(t)^2 \left[ \int_{t_0}^t \alpha_{n_0}(s)^2 \od s + \bar \sigma_{t_0}^2 \right] \\
& = \alpha_{n_0}(t)^2 \left[ \bar \mu_{t_0} - \frac{1}{2}  \bar \sigma_{t_0}^2 \right] \leq 0.
\end{align*}
Again, if either of the inequalities in \eqref{eq_beta_pos} is strict \newer{uniformly for all $t \in [t_0, t_1)$}, then $g(t) < 0$.}

In the same way as above, the condition \eqref{eq_alpha_goes_infinity} together with the boundedness of $\alpha_{n_0}$ implies 
\begin{align*}
\psi(t) & \leq \Phi \left( \frac{ \newer{\left| \rho(0) - \imean - \bar \mu_{t_0} \right|} - \int_{t_0}^t \beta_{n_0}(s) \od s }{\sqrt{ \ivar^2 + \bar \sigma_{t_0}^2 + \int_{t_0}^t \alpha_{n_0}(s)^2 \od s }} \right) \\
& \leq \Phi \left(\frac{ \newer{\left| \rho(0) - \imean - \bar \mu_{t_0} \right|}}{\ivar}   -\frac{\frac{1}{4} \left( \inf_{u \geq t_0} \alpha_{n_0}(u)^2 \right) (t - t_0) }{\sqrt{ \ivar^2 + \bar \sigma_{t_0}^2 + \left( \sup_{u \geq t_0} \alpha_{n_0}(u)^2 \right)(t - t_0) }} \right) \to 0
\end{align*}
as $t \to \infty$, \newer{where we used that $\sup_{t \in \rplus} \rho(t) \leq \rho(0)$, which holds since $\rho$ is non-increasing.}

\end{proof}

\newer{
\begin{remark} \label{rem_when_other_lemma_can_be_used}
In the proof of \cref{theo_existence} below we use \cref{lemma_monotonicity} with $\xi_{t_0} = \widehat Y_{t_0}$, where $\widehat Y$ is a process that satisfies 
\begin{equation} \label{eq_gaussian_sde_special}
\widehat Y_t := F(0, x_0) + \int_0^t \sum_{n=1}^\infty \ind_{\left\{ s \in R_n \right\}} \left[ \alpha_n(s) \od B_s + \beta_n(s) \od s \right], \quad t \in [0, t_0]
\end{equation} 
for some sets $(R_n)_{n=1}^\infty \subset \cB([0, t_0])$ that form a partition of the interval $[0, t_0]$. If $\widehat{Y}$ is a solution to \eqref{eq_gaussian_sde}, then one lets $R_n := \left. \varphi \right|_{[0, t_0]} ^{-1}(I_n)$ for $n \in \N$. In this case a sufficient condition for \eqref{eq_cond_mu} to hold is to assume that \ref{hypo_x0} holds, and for all $n \in \N$ and $t \in \rplus$ one has 
\begin{equation} \label{eq_bounded_beta}
\left| \beta_n(t) \right| \leq \frac{1}{2} \alpha_n(t)^2,
\end{equation}
which is automatically satisfied under \ref{hypo_increasing}. This follows from the following computation: Since by \ref{hypo_x0} the initial value $F(0, x_0)$ in \eqref{eq_gaussian_sde} is normally distributed, we have 
\begin{equation*}
\bar \mu_{t_0} = \E \widehat Y_{t_0} = \imean + \int_0^{t_0}  \sum_{n=1}^\infty \ind_{ \left\{ s \in R_n \right\} } \beta_n(s) \od s 
\end{equation*}
and 
\begin{equation*}
\bar \sigma_{t_0}^2 = \variance(\widehat Y_{t_0}) = \ivar^2 + \int_0^{t_0} \sum_{n=1}^\infty \ind_{ \left\{ s \in R_n \right\} } \alpha_n(s)^2 \od s,
\end{equation*}
which give us the estimate 
\begin{align*}
\left| \bar \mu_{t_0} \right| & \leq \left| \imean \right| + \int_0^{t_0} \sum_{n=1}^\infty \ind_{\left\{ s \in R_n \right\}} \left| \beta_n(s) \right| \od s   \leq \frac{1}{2} \left[ \ivar^2 + \int_0^{t_0} \sum_{n=1}^\infty \ind_{ \left\{ s \in R_n \right\} } \alpha_n(s)^2 \od s \right] = \frac{1}{2} \bar \sigma_{t_0}^2
\end{align*}
by \eqref{eq_cond_mu_initial} and \eqref{eq_bounded_beta}.

\end{remark}
}

\subsection{Proof of \cref{theo_existence}}

Now we may proceed to prove \cref{theo_existence}.

\begin{proof}[Proof of \cref{theo_existence}]
\newer{In the following we assume that $\varphi(0) = \P(x_0 \leq r(0)) \in [y_0, y_1)$. In} the proof we shall see that $\varphi$ is always strictly increasing, \newer{hence in the general case one can always define new levels $(\tilde y_n)_{n=0}^\infty$ by $\tilde y_n := y_{n - m}$, where $m := \max \left\{k \geq 0 \mid \varphi(0) \geq y_k \right\}$. }

\bigskip
{\noindent
\fbox{\textbf{Existence:}} 
By \cref{prop_only_transform} it suffices to construct a solution to the SDE \eqref{eq_gaussian_sde} with the data $\xi_0 = F(0, x_0)$ and $\rho(t) = F(t, r(t))$ and then transform the solution into a solution to \eqref{eq_target_sde} using the transformation $G$ defined in \eqref{eq_def_G_initial}.
}

We define
\begin{equation*}
Y_t^1 := F(0, x_0) + \int_0^t \alpha_1(s) \od B_s + \int_0^t \beta_1 (s) \od s, \quad t \geq 0.
\end{equation*}
Using \cref{lemma_monotonicity} \ref{item_increasing}  with $t_0 = 0$ and $\xi_{t_0} = F(0, x_0)$ we get that the function $\varphi_1(t) := \P \left( Y_{t}^1 \leq \rho(t) \right)$, $t \in \rplus$, is \newer{continuous and} strictly increasing. If we define $T_1 := \inf \left\{ t > T_0 \mid \varphi_1(t) = y_1 \right\} \in (0, \infty]$, where $T_0 := 0$, then $\varphi_1(t) \in [y_0, y_1)$ if and only if $t \in [0, T_1)$, hence the process $Y^1$ solves \eqref{eq_gaussian_sde} on $[0, T_1)$.

If $T_1 = \infty$, then we let $\tmax = T_1 = \infty$ and stop. Otherwise, for an arbitrary $n > 1$ we define
\begin{equation} \label{eq_linear_n}
Y_t^n := Y_{T_{n-1}}^{n-1} + \int_{T_{n-1}}^{t \vee T_{n-1} } \alpha_n(s) \od B_s + \int_{T_{n-1}}^{t \vee T_{n-1}} \beta_n (s) \od s, \quad t \geq 0,
\end{equation}
where $T_{n - 1} = \inf \left\{ t > T_{n-2} \mid \varphi_{n-1}(t) = y_{n - 1} \right\} \in (T_{n-2}, \infty]$. Assuming that $T_{n-1} < \infty$ we apply \cref{lemma_monotonicity} \ref{item_increasing}  with $t_0 = T_{n-1}$ and $\xi_{t_0} = Y_{T_{n-1}}^{n-1}$ \newer{(which is possible by \cref{rem_when_other_lemma_can_be_used})} to obtain that
\begin{equation} \label{eq_varphi_n}
\varphi_n(t) := \P \left( Y_{t}^n \leq \rho(t) \right)
\end{equation}
is \newer{continuous and} strictly increasing on $[T_{n-1}, \infty)$, and let
\begin{equation} \label{eq_def_T_n}
T_n := \inf \left\{ t > T_{n-1} \mid \varphi_n(t) = y_n \right\} \in (T_{n-1}, \infty].
\end{equation}
Since $\varphi_n(t) \in [y_{n-1}, y_n)$ if and only if $t \in [T_{n-1}, T_n)$, the process $Y^n$ solves \eqref{eq_gaussian_sde} on $[0, T_{n})$. 

If $T_n = \infty$, then we let $\tmax = T_n = \infty$ and stop, otherwise we continue the procedure for $n + 1$. In the end we obtain an increasing sequence $0 =: T_0 < T_1 < ... < \tmax \leq \infty$, where $\tmax = \infty$ if $T_n = \infty$ for \corr{some} $n \in \N$, otherwise we let $\tmax := \lim_{n \to \infty} T_n$. We let $N := \inf \left\{ n \in \N \mid T_n = \infty \right\} \in \N \cup \left\{ \infty \right\}$ and define a process 
\begin{equation*}
Y_t := F(0, x_0) + \int_0^t \sum_{n=1}^N \ind_{\left\{ s \in [T_{n-1}, T_n) \right\}} \left[ \alpha_n(s) \od B_s + \beta_n(s) \od s \right]
\end{equation*}
for $t \in [0, \tmax)$, which, by the previous computations, solves \eqref{eq_gaussian_sde} on $[0, \tmax)$. Finally, by \cref{prop_only_transform} the process $X_t := G(t, Y_t)$ solves \eqref{eq_target_sde} on $[0, \tmax)$ and $X_t \in U_t$ for all $t \in [0, \tmax)$.

\bigskip

\newer{{\noindent {\color{black} \fbox{\textbf{Uniqueness:}}} Let us suppose that there exists a second solution $Z = (Z_t)_{t \in [0, S)}$ to \eqref{eq_gaussian_sde} with the data $\xi_0 = F(0, x_0)$ and $\rho(\cdot) = F(\cdot, r(\cdot))$ for some $S \in (0, \infty]$. 
Let $\psi(t) := \P(Z_{t} \leq \rho(t))$ for $t \in [0, S)$, and define 
\begin{equation*}
S_1 := \inf \left\{ t \in [0, S) \mid \psi(t) = y_1 \right\} \wedge S.
\end{equation*}
Since in the beginning of this proof it was assumed that
\begin{equation*}
\psi(0) = \P(\xi_0 \leq \rho(0)) = \P(F(0, x_0) \leq F(0, r(0))) = \P(x_0 \leq r(0)) < y_1,
\end{equation*}
it holds $S_1 > 0$ by the continuity of $\psi$ (the continuity is proven similarly as in the proof of \cref{lemma_monotonicity}). Furthermore, since $Y$ and $Z$ share the same initial value, it holds $\psi(0) = \varphi(0)$ and both processes use the same coefficient functions $\alpha_1$ and $\beta_1$ on $[0, T_1 \wedge S_1)$. Thus, $\varphi(t) = \psi(t)$ and $\P(Z_t = Y_t) = 1$ for all $t \in [0, T_1 \wedge S_1)$. In particular, if $S \geq S_1$, then $S_1 = T_1$.}

Next let us assume that $S > S_1$ and define 
\begin{equation*}
S_2 := \inf \left\{ t \geq 0 \mid \psi(t) = y_2 \right\} \wedge S.
\end{equation*}
We want to show that $\psi(t) \geq y_1$ for all $t \in [S_1, S_2)$. Let us assume the opposite: there is some $u_1 \in (S_1, S_2)$ such that $\psi(u_1) < y_1$. By the continuity of $\psi$ there is a $u_0 \in [S_1, u_1)$ such that $\psi(u_0) = y_1$ and $\psi(t) < y_1$ for all $t \in (u_0, u_1]$. However, by \cref{lemma_monotonicity} \ref{item_increasing} (and \cref{rem_when_other_lemma_can_be_used}) the function $\psi$ is strictly increasing on $(u_0, u_1)$, which is a contradiction.

The property $\psi(t) \geq y_1$ for $t \in [S_1, S_2)$ now implies that $Z$ uses only $\alpha_2$ and $\beta_2$ on $[S_1, S_2)$. \newer{Thus, it follows from} $\psi(S_1) = \varphi(S_1)$ \newer{that} $\psi(t) = \varphi(t)$ and $\P(Z_t = Y_t) = 1$ for $t \in [S_1, T_2 \wedge S_2)$. Especially if $S \geq S_2$, then $S_2 = T_2$.} Continuing inductively one can show that $T_n = S_n$ for all \newer{$n \in \N$ that satisfy $S_n \leq S$}, and consequently $S \leq \tmax$ and $Z_t = Y_t$ $\P$-almost surely for all $t \in [0, S)$.

It remains to show that the uniqueness also holds for the original equation \eqref{eq_target_sde}. \corr{Supposing any solution $\widetilde{X}$ to \eqref{eq_target_sde} with a lifetime $S > 0$ such that $\widetilde{X}_t \in U_t$ for $t \in [0, S)$, we observe that $\widetilde{Y} := F(t, \widetilde{X}_t)$ solves \eqref{eq_gaussian_sde} by \cref{prop_only_transform}.}
\corr{But now by the previous arguments it holds} $Y_t = \widetilde{Y}_t$ $\P$-almost surely for $t \in [0, \tmax \wedge S)$, \corr{
which in turn implies $X_t = G(t, Y_t) = G(t, \widetilde{Y}_t) = \widetilde{X}_t$ $\P$-almost surely for all $t \in [0, \tmax \wedge S)$. Furthermore, one has $S \leq \tmax$.}

\bigskip

{\noindent \fbox{\textbf{Globality:}} If $T_n = \infty$ for some $n \in \N$, then we are done, so let us assume that $T_n < \infty$ for all $n \in \N$. We let 
\begin{equation} \label{eq_def_funky_A_B}
\mathscr{A}(t) := \sum_{n=1}^{\infty} \ind_{\left\{ \varphi(t) \in [y_{n-1}, y_{n}) \right\}} \alpha_n(t), \qquad \mathscr{B}(t) := \sum_{n=1}^{\infty} \ind_{\left\{ \varphi(t) \in [y_{n-1}, y_{n}) \right\}} \beta_n(t)
\end{equation}
for $t \in [0, \tmax)$. To show that \eqref{eq_as_bounded_beta} implies $\tmax = \infty$ we apply \eqref{eq_cond_strong_diffusion}, \eqref{eq_as_bounded_beta} and \ref{hypo_r} to get the estimate
\begin{align*}
y_n = \varphi(T_n) = \Phi \left( \frac{\rho(T_n) - \imean - \int_0^{T_n} \mathscr{B}(u) \od u }{\sqrt{ \ivar^2 + \int_0^{T_n} \mathscr{A}(u)^2 \od u }} \right) \leq \Phi \left( \frac{\int_0^{T_n} \mathscr{A}(u)^{\newer{2}} \od u }{2 \sqrt{ \ivar^2 + \int_0^{T_n} \mathscr{A}(u)^2 \od u }} \right) \leq \Phi \left( \frac{C}{2 \ivar} {T_n} \right),
\end{align*}
which implies
\begin{equation*}
T_n \geq \frac{2 \ivar}{C}  \Phi^{-1}(y_n) \to \infty
\end{equation*}
as $n \to \infty$ since $\lim_{n \to \infty} y_n = 1$. 
}

\bigskip

{\noindent \fbox{\textbf{Maximality:}} Let us suppose that $\tmax < \infty$. 
We \newer{assume} that there is an $\ep > 0$ such that \eqref{eq_gaussian_sde} has a solution on $[0, \tmax + \ep)$. We define $Z_0 := F(0, x_0)$ and
\begin{equation*}
Z_n := \int_{T_{n-1}}^{T_n} \left[ \alpha_n(s) \od B_s + \beta_n(s) \od s \right], \quad n \in \N,
\end{equation*}
which are independent Gaussian random variables. We observe that we can write $Y_{T_n} = M_n$, where $\corr{M_n} := \sum_{k=0}^n Z_k$, for all $n \in \N$. By the continuity of the sample paths there is a random variable $M$ such that $M_n \to M$ almost surely.
Furthermore, this implies that the \corr{sequence of} characteristic functions of \corr{$(M_n)_{n=1}^\infty$} converges pointwise:
\begin{equation*}
\psi_{M_n}(u) := \exp\left( i u \imean - \frac{u^2 \ivar^2}{2} \right) \prod_{k=1}^n \exp \left( i u \int_{T_{k-1}}^{T_k} \beta_k(s) \od s - \frac{u^2}{2} \int_{T_{k-1}}^{T_k} \alpha_k(s)^2 \od s \right) \to \psi_M(u),
\end{equation*}
for all $u \in \R$ as $n \to \infty$, where \newer{$\psi_M$} is the characteristic function of $M$.}

Using the notation \eqref{eq_def_funky_A_B} we have
\begin{equation*}
\psi_{M_n}(u) = \exp\left[ i u \left(\imean + \int_0^{T_n} \mathscr{B}(s) \od s \right) - \frac{u^2}{2} \left( \ivar^2 + \int_0^{T_n} \mathscr{A}(s)^2 \od s \right) \right]
\end{equation*}
for all $n \in \N$ and $u \in \R$. The identity
\begin{equation*}
\varphi(T_n) = \Phi \left( \frac{\rho(T_n) - \imean -  \int_0^{T_n} \mathscr{B}(s) \od s }{\sqrt{\ivar^2 + \int_0^{T_n} \mathscr{A}(s)^2 \od s}} \right) = y_n
\end{equation*}
implies that
\begin{align} \label{eq_estimate_a_infinity}
- \imean - \int_0^{T_n} \mathscr{B}(s) \od s & = \left( \sqrt{\ivar^2 + \int_0^{T_n} \mathscr{A}(s)^2 \od s} \right) \Phi^{-1}(y_n) - \rho(T_n)  \geq \ivar \Phi^{-1}(y_n) - \imean
\end{align}
for $n \geq \min \left\{ n \in \N \mid y_n \geq \frac{1}{2} \right\}$ by \ref{hypo_r} \newer{(for smaller $n$ the quantity $\Phi^{-1}(y_n)$ is negative, so the last inequality in \eqref{eq_estimate_a_infinity} does not necessarily hold)}, hence $\int_0^{T_n} \mathscr{B}(s) \od s \to -\infty$ as $n \to \infty$ since $y_n \uparrow 1$. But as a result we get by using \ref{hypo_increasing} that
\begin{equation*}
\ivar^2 + \int_0^{T_n} \mathscr{A}(s)^2 \od s \geq \ivar^2 - 2 \int_0^{T_n}  \mathscr{B}(s) \od s \to \infty
\end{equation*}
as $n \to \infty$. Consequently,
\begin{align*}
\left| \psi_{M_n}(u) - 0 \right| = \exp \left[ -\frac{u^2}{2} \left( \corr{\ivar^2} + \int_0^{T_n} \mathscr{A}(s)^2 \od s \right) \right] \to 0
\end{align*}
\newer{for all $u \neq 0$} as $n \to \infty$, \newer{so since $\psi_M(0) = 1$ it should hold $\psi_M(u) = \ind_{\left\{u = 0 \right\}}$ for all $u \in \R$. This is a contradiction since the characteristic function of a random variable has to be continuous.} We conclude that $Y_{T_n} = M_n$ does not converge almost surely to any random variable, hence there is no solution to \eqref{eq_gaussian_sde} --- and consequently to \eqref{eq_target_sde} either --- on $[0, \tmax + \ep)$ for any $\ep > 0$.

\bigskip

{\noindent \fbox{\textbf{Bijectivity:}} 
From the construction of the process $Y$ \new{in \newer{the} "Existence" step above} \corr{one sees that $\varphi(t) = \varphi_k(t)$ for $t \in [t_{k-1}, t_k)$ and all $k \in \N$ such that $T_k < \infty$, \new{where $\varphi_k$ is defined in \eqref{eq_varphi_n}}. Since each function $\varphi_k$ was proven to be \newer{continuous and} strictly increasing, the function $\varphi$ is also \newer{continuous and} strictly increasing. Assuming that \eqref{eq_positive_alphas} holds, then by \cref{lemma_monotonicity} one has 
\begin{equation*}
\lim_{t \to \infty} \varphi_k(t) = \infty
\end{equation*}
for all $k \in \N$ with $T_k < \infty$. Consequently,
\begin{equation*}
\inf \left\{ t \in \rplus \mid \varphi(t) \geq y_n \right\} = \inf \left\{ t \in \rplus \mid \varphi_n(t) = y_n \right\} < \infty
\end{equation*}
\newer{for all $n \in \N$,} which completes the proof.
}}

\end{proof}

In the example below we demonstrate that without condition \eqref{eq_as_bounded_beta} it can occur that the solution obtained in \cref{theo_existence} only exists for a finite time. However, by choosing different levels $(y_n)_{n=1}^\infty$ we can make the solution global without changing the functions $\alpha_n$.

\begin{example} \label{ex_explosion}
Consider the linear SDE
\begin{equation} \label{eq_linear_sde_explosition}
\begin{cases}
\begin{aligned}
X_t & = x_0 + \int_0^t \left( \sum_{n=1}^{\infty} \ind_{ \left\{ \varphi(s) \in [y_{n-1}, y_n) \right\} } \newest{\sqrt{2n}} \right)  X_s \od B_s, \quad t \in \rplus, \\
\varphi(t) & = \P( X_t \leq \mathrm{median}(x_0)),
\end{aligned}
\end{cases}
\end{equation}
where $x_0 \sim \lognormal(\newer{0, 1})$. \newer{Here $\alpha_n \equiv \newest{\sqrt{2n}}$ for all $n \in \N$, $\sigma_1 \equiv 1$, $\sigma_2 \equiv 0$ and \newest{$b \equiv 0$}. We observe that} $\mathrm{median}(x_0) = \exp(\newer{0}) \newer{= 1}$, thus by using \cref{prop_only_transform} and \cref{ex_linear_SDE} we see that the transformation $F(t, x) := \log(x)$ transforms the equation \eqref{eq_linear_sde_explosition} into the SDE 
\begin{equation} \label{eq_sde_explosition}
\begin{cases}
\begin{aligned}
Y_t & = \log(x_0) + \int_0^t \sum_{n=1}^{\infty} \ind_{ \left\{ \varphi(s) \in [y_{n-1}, y_n) \right\} } \left[ \newest{\sqrt{2n}} \od B_s - n \od s \right], \quad t \in \rplus, \\
\varphi(t) & = \P( Y_t \leq \newer{0}).
\end{aligned}
\end{cases}
\end{equation}
Since now $\beta_n \equiv -n$ for all $n \in \N$, the assumption \ref{hypo_increasing} holds, and therefore by \cref{theo_existence} we know that \eqref{eq_linear_sde_explosition} has a unique strong solution defined on some maximal interval $[0, \tmax) \subseteq \rplus$.

However, the family $(\alpha_n)_{n \in \N}$ is unbounded, so the condition \eqref{eq_as_bounded_beta} does not hold. We show that whether a global solution exists depends on the choice of $(y_n)_{n=1}^\infty$. \newer{Assume a strictly increasing sequence of time points $(t_n)_{n=0}^\infty$ such that $t_0 = 0$ and}
\begin{equation} \label{eq_finite_sum}
\sum_{k=1}^{\infty} (t_{k} - t_{k-1}) k = \infty.
\end{equation}
\newest{
Let us define a process
\begin{equation*} 
\widehat Y_t := \log(x_0) + \int_0^t \sum_{n=1}^{\infty} \ind_{ \left\{ s \in [t_{n-1}, t_n) \right\} } \left[  \newest{\sqrt{2n}} \od B_s - n \od s \right], \quad t \in \rplus,
\end{equation*}
and let $\widehat \varphi(t) := \P(\widehat Y_t \leq 0)$. With the help of the functions 
\begin{align*}
& f(t) := \int_{0}^t \left( \sum_{n=1}^\infty \ind_{\left\{ s \in [t_{n-1}, t_n) \right\}} n \right) \od s, \hspace{-3.0cm} && t \in \rplus, \\
& g(x) := \frac{x}{\sqrt{1 + 2x}}, && x \in \rplus,
\end{align*}
one can write
\begin{equation*}
\widehat \varphi(t) = \Phi\left(
\frac{\int_{0}^t (\sum_{n=1}^\infty \ind_{\left\{ s \in [t_{n-1}, t_n) \right\}} n) \od s}{ \sqrt{1 + 2 \int_{0}^t (\sum_{n=1}^\infty \ind_{\left\{ s \in [t_{n-1}, t_n) \right\}} n) \od s} }\right) = \Phi((g \circ f)(t))
\end{equation*}
for $t \in \rplus$. The functions $f$ and $g$ have the following properties:
\begin{enumerate}[label=\rm{(\roman*)}]
    \item $f$ and $g$ are strictly increasing,
    \item $g(x) \to \infty$ as $x \to \infty$, and 
    \item for all $n \in \N$ it holds
    \begin{equation*}
    f(t_n) = \sum_{k=1}^\infty \int_{0}^{t_n} \ind_{\left\{ s \in [t_{k-1}, t_k) \right\}} k \od s = \sum_{k=1}^n \int_{t_{k-1}}^{t_k} k \od s = \sum_{k=1}^n (t_k - t_{k-1}) k \to \infty
    \end{equation*}
    as $n \to \infty$ by \eqref{eq_finite_sum}
\end{enumerate}
Recalling that $\Phi : \R \to (0, 1)$ is a strictly increasing and continuous bijection, these properties imply that the composition $\widehat \varphi = \Phi \circ g \circ f$ is strictly increasing. \changed{Moreover, it also follows from the properties above that} $\widehat \varphi(t_n) \to 1$ as $n \to \infty$. Consequently, the sequence 
\begin{equation*}
y_n := 
\begin{cases}
\begin{aligned}
& 0, \quad && n = 0, \\
& \widehat \varphi(t_n), \quad && n \geq 1,
\end{aligned}
\end{cases}
\end{equation*}
satisfies $0 = y_0 < y_1 < y_2 < ... \uparrow 1$.

Using the monotonicity of $\widehat \varphi$ one sees that $\widehat \varphi(t) \in [y_{n-1}, y_n)$ if and only if $t \in [t_{n-1}, t_n)$, which in fact implies that $\widehat Y$ solves \eqref{eq_sde_explosition}}. Furthermore, 
\begin{equation*}
\inf \left\{ t \in \rplus \mid \widehat \varphi(t) = y_n \right\} = t_n
\end{equation*}
for all $n \in \N$, hence $\tmax = \lim_{n \to \infty} t_n$. Thus, whether a global solution exists depends on the choice of $(t_n)_{n=0}^\infty$.
For example, if one chooses $t_n = \frac{n}{n + 1}$, then 
\begin{equation*}
\sum_{n=1}^\infty (t_{n} - t_{n-1})n = \sum_{n=1}^\infty \frac{1}{n + 1} \newest{= \infty},
\end{equation*}
hence \eqref{eq_finite_sum} holds, and $\tmax = \lim_{n \to \infty} \frac{n}{n + 1} = 1 < \infty$. To obtain a global solution one can choose for instance $t_n = n$: the condition \eqref{eq_finite_sum} clearly holds, and now $\tmax = \lim_{n \to \infty} n = \infty$.

\end{example}

\section{Equations with multiple solutions or no (global) solution} \label{sec_counter_examples}

Violating the assumption \ref{hypo_increasing} has certain side effects: we might lose the uniqueness or even the existence of a (global) solution. \newest{For simplicity, we study these effects for the transformed equation \eqref{eq_gaussian_sde}.}

Supposing that $I_n = [y_{n-1}, y_n)$ for all $n \in \N$, then the possible "bad" behavior is always localized in a neighborhood of some "critical" level $y_{n}$ by the continuity of $\varphi$. Inspired by this we therefore restrict ourselves to study equations of the type 
\newest{
\begin{equation} \label{eq_sde_one_level}
\begin{cases}
\begin{aligned}
Y_t & = \xi_{t_0} + \int_{t_0}^t \ind_{\left\{ \varphi(s) < y \right\}} \left[ \alpha_1(s) \od B_s + \beta_1(s) \od s \right] \\
& \hspace{0.9cm} + \int_{t_0}^t \ind_{\left\{ \varphi(s) \geq y \right\}} \left[ \alpha_2(s) \od B_s + \beta_2(s) \od s \right], \quad t \in [t_0, \infty),   \\
\varphi(t) & = \P(Y_t \leq \imean),
\end{aligned}
\end{cases}
\end{equation}
where:
\begin{enumerate}[label=\rm{(\roman*)}]
    \item $t_0 \in \rplus$ and $y \in (0, 1)$,
    \item $\alpha_1, \alpha_2 : \rplus \to \rplus$ and $\beta_1, \beta_2 : \rplus \to \R$ are continuous and bounded, and 
    \item $\xi_{t_0}$ is an $\cF_{t_0}$-measurable random variable satisfying the following: there exist parameters $(\imean, \bar \mu_{t_0}, \ivar^2, \bar \sigma_{t_0}^2) \in \R \times \R \times (0, \infty) \times [0, \infty)$ such that 
\begin{equation*}
\left| \bar \mu_{t_0} \right| \leq \frac{1}{2} \bar \sigma_{t_0}^2,
\end{equation*}
$\xi_{t_0} \sim \cN(\imean + \bar \mu_{t_0}, \ivar^2 + \bar \sigma_{t_0}^2)$ and $\P(\xi_{t_0} \leq \imean) = y$.
\end{enumerate}
In equation \eqref{eq_sde_one_level} one has $I_1 = [0, y)$, $I_2 = [y, \infty)$ and $I_n = \emptyset$ for $n > 2$.

\smallskip
}

First we give conditions that ensure that equation \eqref{eq_sde_one_level} has exactly two solutions.
\begin{proposition} \label{prop_many_solutions}
If 
\begin{align}  
\frac{1}{4} \alpha_{1}(t)^2 \leq & \, \beta_{1}(t) \leq \frac{1}{2} \alpha_{1}(t)^2, \label{eq_beta_pos_1} \\
-\frac{1}{2} \alpha_{2}(t)^2 \leq & \, \beta_{2}(t) \leq -\frac{1}{4} \alpha_{2}(t)^2 \label{eq_beta_neg_2}
\end{align}
holds for all $t \in [t_0, \infty)$ and in both lines at least one of the inequalities is strict \newer{uniformly in $t$}, then the equation \eqref{eq_sde_one_level} has exactly two solutions defined on $[t_0, \infty)$.
\end{proposition}

\begin{proof}
We claim that the only solutions are 
\begin{equation*}
Y_t^i := \newest{\xi_{t_0}} + \int_{t_0}^t \alpha_i(s) \od B_s + \int_{t_0}^t \beta_i(s) \od s, \quad t \in [t_0, \infty)
\end{equation*}
for $i \in \left\{1, 2\right\}$. We denote $\varphi_i(t) := \P(Y_t^i \leq \imean)$ for all $t \in [t_0, \infty)$ and $i \in \left\{ 1, 2\right\}$. By \cref{lemma_monotonicity} we see that $\varphi_1$ is strictly decreasing and $\varphi_2$ is strictly increasing on $[t_0, \infty)$. Since also $\varphi_1(t_0) = y$, the process $Y^1$ solves \eqref{eq_sde_one_level}. Moreover, by observing that we also have $\varphi_2(t_0) = y$ and the set $\left\{ t \geq 0 \mid \varphi_2(t) \leq y \right\} = \left\{ 0 \right\}$ has Lebesgue measure zero, we see that the process $Y^2$ is another solution to \eqref{eq_sde_one_level}.

Let us suppose there is another solution $Y^3$ on $[t_0, T)$ for some $T \in (t_0, \infty]$, and let us define $\varphi_3(t) := \P(Y_{t}^3 \leq \imean)$ for $t \in [t_0, T)$. \newer{Then there exist $t_0 < t_1 < t_2 < T$ such that either $\varphi_3(t_1) < y$ and $\varphi_3(t_2) \geq y$, or $\varphi_3(t_1) \geq y$ and $\varphi_3(t_2) < y$, because the function $\varphi_3$ cannot agree with $\varphi_1$ or $\varphi_2$ everywhere on $[t_0, T)$. Let us consider both cases separately:  
\begin{itemize}
    \item \underline{\textbf{Case 1:} $\varphi_3(t_1) < y$ and $\varphi_3(t_2) \geq y$:}
    
    By the continuity of $\varphi_3$ there is an $s \in (t_1, t_2]$ such that $\varphi_3(s) = y$ and $\varphi_3(t) < y$ for all $t \in (t_1, s)$. However, by \eqref{eq_beta_pos_1} and \cref{lemma_monotonicity} \ref{item_decreasing} this makes the function $\varphi_3$ strictly decreasing on $(t_1, s)$, which is a contradiction.

    \item \underline{\textbf{Case 2:} $\varphi_3(t_1) \geq y$ and $\varphi_3(t_2) < y$:}

    Let us first suppose that $\varphi_3(t_1) > y$. Now there is an $s \in (t_1, t_2)$ such that $\varphi_3(s) = y$ and $\varphi_3(t) > y$ for all $t \in (t_1, s)$. But \eqref{eq_beta_neg_2} and \cref{lemma_monotonicity} \ref{item_increasing} imply that $\varphi_3$ is strictly increasing on $(t_1, s)$, which is again a contradiction.

    Next let us assume that $\varphi_3(t_1) = y$. If there is a $u \in (t_0, t_1)$ such that $\varphi_3(u) < y$, then this corresponds with \textbf{Case 1} by choosing \newest{$\tilde t_1 = u$ and $\tilde t_2 = t_1$}, so we can assume that $\varphi_3(t) \geq y$ for all $t \in [t_0, t_1]$. However, by \eqref{eq_beta_neg_2} and \cref{lemma_monotonicity} \ref{item_increasing} the function $\varphi_3$ is now strictly increasing on $(t_0, t_1)$, which implies $y = \varphi_3(t_0) < \varphi_3(t_1) = y$, which gives the remaining contradiction. 

\end{itemize}}
We conclude that $Y^3 = Y^i$ almost surely on $[t_0, T)$ for some $i \in \left\{1, 2\right\}$, hence there are exactly two solutions to \eqref{eq_sde_one_level}.
\end{proof}

\begin{remark} \label{rem_inf_solutions}
If neither of the inequalities in \eqref{eq_beta_pos_1} is strict, then there can exist infinitely many solutions: Assume that
\begin{equation} \label{eq_beta_const_1}
0 = \alpha_{\newest{2}}(t) = \beta_{\newest{2}}(t)
\end{equation}
holds for all $t \in [t_0, \infty)$. Then \eqref{eq_beta_neg_2} is still satisfied, but there is an equality on both sides. Let us suppose that the other assumptions of \cref{prop_many_solutions} hold. Then the equation \eqref{eq_sde_one_level} has infinitely many solutions: for each $\newer{w \in [0, \infty)}$ the process 
\begin{equation*}
Y_t^w := \newest{\xi_{t_0}} + \int_{t_0}^t \ind_{ \left\{ s \in [w, \infty) \right\} } \left[ \alpha_{\newest{1}}(s) \od B_s + \beta_{\newest{1}}(s) \od s \right], \quad t \in [t_0, \infty)
\end{equation*}
is a solution \newer{to \eqref{eq_sde_one_level}}. This \newest{is due to the following observation: For all $t \in [t_0 + w, \infty)$ it holds $\varphi_w(t) := \P(Y_{t}^w \leq \imean) = y$}, and on $[t_0 + w, \infty)$ the function $\varphi_w$ is strictly decreasing by \cref{lemma_monotonicity}, \newest{hence $\varphi_w(t) < y$ for all $t > t_0 + w$ by continuity}.
\end{remark}

In the following proposition we show that without the condition \ref{hypo_increasing} the equation \eqref{eq_sde_one_level} can also fail to have a solution due to an oscillating behavior of $\varphi$. This is comparable to \cite[Proposition 3.3]{paper1}, where a similar effect is studied for $\varphi(t) = \E \left| X_t - z \right|^p$. \newest{This effect is attained by swapping the functions $(\alpha_1, \beta_1)$ and $(\alpha_2, \beta_2)$ in conditions \eqref{eq_beta_pos_1} and \eqref{eq_beta_neg_2}, which inverts the direction of the monotonicity of $\varphi$ on the sets $[0, y)$ and $[y, \infty)$, implying that whenever the function $\varphi$ leaves the level $y$, then it is immediately driven back.}

\begin{proposition} \label{theo_no_solution}
If there is a $\delta > 0$ such that 
\begin{align}  
-\frac{1}{2} \alpha_{1}(t)^2 \leq & \, \beta_{1}(t) \leq -\frac{1}{4} \alpha_{1}(t)^2, \label{eq_beta_neg_1} \\
\frac{1}{4} \alpha_{2}(t)^2 \leq & \, \beta_{2}(t) \leq \frac{1}{2} \alpha_{2}(t)^2 \label{eq_beta_pos_2}
\end{align}
holds for all $t \in [t_0, t_0 + \delta)$ and in both lines at least one of the inequalities is strict \newer{uniformly in $t$}, then there is no $\ep \in (0, \delta)$ such that the equation \eqref{eq_sde_one_level} has a solution on the interval $[t_0, t_0 + \ep)$.
\end{proposition}

\begin{proof}

\newest{Suppose} that there is an $\ep \in (0, \delta)$ such that a solution to \eqref{eq_sde_one_level} exists on $[t_0, t_0 + \ep)$. Let us choose any $t_1 \in (t_0, t_0 + \ep)$. Then there are three different possible cases that can occur:
\begin{itemize}
    \item \newer{\underline{\textbf{Case 1:} $\varphi(t_1) > y$:}}
    
    By the continuity of $\varphi$ there is an $s \in [t_0, t_1)$ such that $\varphi(s) = y$ and $\varphi(t) > y$ for all $t \in (s, t_1]$. However, by \eqref{eq_beta_pos_2} and \cref{lemma_monotonicity} \ref{item_decreasing} the function $\varphi$ is now strictly decreasing on $(s, t_1)$, which is a contradiction since \newer{$\varphi$ is continuous and} we assumed that \newer{$\varphi(s) < \varphi(t_1)$ and $s < t_1$.}
    
    \item \newer{\underline{\textbf{Case 2:} $\varphi(t_1) < y$:}}
    
    In this case there is an $s \in [t_0, t_1)$ such that $\varphi(s) = y$ and $\varphi(t) < y$ for all $t \in (s, t_1]$, but by \eqref{eq_beta_neg_1} and \cref{lemma_monotonicity} \ref{item_increasing} this implies that the function $\varphi$ is strictly increasing on $(s, t_1)$, thus giving another contradiction.
    
    \item \newer{\underline{\textbf{Case 3:} $\varphi(t) = y$ for all $t \in [t_0, t_1]$:}}
    
    Since now \newer{$\varphi(t) \in [y, \infty)$} for all $t \in [t_0, t_1]$, \newer{the function $\alpha_2$ is used by the diffusion coefficient on $[t_0, t_1]$, hence} \eqref{eq_beta_pos_2} and \cref{lemma_monotonicity} \ref{item_decreasing} imply that $\varphi$ is strictly decreasing on $[t_0, t_1]$, which is \newer{again} a contradiction.

\end{itemize}
We conclude that there is no solution to \eqref{eq_sde_one_level} on any interval $[t_0, t_0 + \ep)$.

\end{proof}

To illustrate a situation when an oscillation effect described in \cref{theo_no_solution} can occur we consider the following simple example.
\begin{example}
\changed{
Consider the equation 
\begin{equation} \label{eq_funny_example_1} 
\begin{cases}
\begin{aligned}
X_t & = x_0 + \int_0^t \left[ \ind_{\left\{ \varphi(s) < \newer{\frac{1}{2}} \right\}} \newer{\alpha_1} + \ind_{\left\{ \varphi(s) \geq \newer{\frac{1}{2}} \right\}} \newer{\alpha_2} \right]  X_s \od B_s + \int_0^t b X_s \od s, \quad t \in \rplus, \\
\varphi(t) & = \P(X_s \leq \textrm{median}(x_0)),
\end{aligned}
\end{cases}
\end{equation}
where $x_0 \sim \newer{\lognormal(0, 1)}$ and \newer{$\alpha_1, \alpha_2$ and $b$} \newer{are constants such that
\begin{equation} \label{eq_funny_conditions}
\begin{aligned}
& 0 < b < \frac{1}{4} \alpha_1^2, \\
& 0 < \frac{3}{4} \alpha_2^2 < b < \alpha_2^2.
\end{aligned}
\end{equation}}
Here $\sigma_1 \equiv 1, \sigma_2 \equiv 0$ and $b(t, x) = b x$. We recall that since $x_0 \sim \newer{\lognormal(0, 1)}$, it holds $\textrm{median}(x_0) = \exp(0) = 1$. By \cref{ex_linear_SDE} the transformation $F(t, x) = \log(x)$ transforms SDE \eqref{eq_funny_example_1} into 
\begin{equation} \label{eq_funny_example_2}
\begin{cases}
\begin{aligned}
Y_t & = \log(x_0) + \int_{0}^t \ind_{\left\{ \varphi(s) < \frac{1}{2} \right\}} \left[ \alpha_1(s) \od B_s + \beta_1 \od s \right] \\
& \hspace{1.55cm} + \int_{0}^t \ind_{\left\{ \varphi(s) \geq \frac{1}{2} \right\}} \left[ \alpha_2(s) \od B_s + \beta_2 \od s \right], \quad t \in [0, \infty),   \\
\varphi(t) & = \P(Y_t \leq 0),
\end{aligned}
\end{cases}
\end{equation}
where
\begin{equation*}
\beta_n = b - \frac{1}{2} \alpha_n^2
\end{equation*}
for $n \in \left\{1, 2\right\}$. Since $\log(x_0) \sim \cN(0, 1)$, we have $\P(\log(x_0) \leq 0) = \frac{1}{2}$. We see that equation \eqref{eq_funny_example_2} is a special case of \eqref{eq_sde_one_level} with $y = \frac{1}{2}$, $t_0 = 0$, \changedagain{$\bar \mu_{t_0} = \bar \sigma_{t_0} = 0$}, $\xi_{\changedagain{t_0}} = \log(x_0)$, $\imean = 0$ and $\ivar^2 = 1$. Furthermore, condition \eqref{eq_funny_conditions} now implies
\begin{equation*}
\begin{aligned}
-\frac{1}{2} \alpha_1^2 & < b - \frac{1}{2} \alpha_1^2 = \beta_1 < -\frac{1}{4} \alpha_1^2, \\
\frac{1}{4} \alpha_2^2 & < b - \frac{1}{2} \alpha_2^2 = \beta_2 < \frac{1}{2} \alpha_2^2,
\end{aligned}
\end{equation*}
which means that conditions \eqref{eq_beta_neg_1} and \eqref{eq_beta_pos_2} hold. We conclude by \cref{theo_no_solution} that equation \eqref{eq_funny_example_2} has no solution on $[0, \delta)$ for any $\delta > 0$.
}
\end{example}


\small
\begin{thebibliography}{10}


\bibitem{bauer_irregular_drift}Bauer, M., Meyer-Brandis, T. \& Proske, F. Strong solutions of mean-field stochastic differential equations with irregular drift. {\em Electron. J. Probab.}. \textbf{23} pp. Paper No. 132, 35 (2018), https://doi.org/10.1214/18-EJP259


\bibitem{garzon_discontinuous_fractional}Garzón, J., León, J. \& Torres, S. Fractional stochastic differential equation with discontinuous diffusion. {\em Stoch. Anal. Appl.}. \textbf{35}, 1113-1123 (2017), https://doi.org/10.1080/07362994.2017.1358643

\bibitem{huang_discont_wasserstein}Huang, X. \& Wang, F. McKean-Vlasov SDEs with drifts discontinuous under Wasserstein distance. {\em Discrete Contin. Dyn. Syst.}. \textbf{41}, 1667-1679 (2021), https://doi.org/10.3934/dcds.2020336


\bibitem{lejay_discontinuous_media}Lejay, A. \& Pichot, G. Simulating diffusion processes in discontinuous media: a numerical scheme with constant time steps. {\em J. Comput. Phys.}. \textbf{231}, 7299-7314 (2012), https://doi.org/10.1016/j.jcp.2012.07.011

\bibitem{leobacher_discont_drift}Leobacher, G., Reisinger, C. \& Stockinger, W. Well-posedness and numerical schemes for one-dimensional McKean--Vlasov equations and interacting particle systems with discontinuous drift. {\em BIT Numerical Mathematics}. (2022,5), https://doi.org/10.1007/s10543-022-00920-4

\bibitem{disc_drift_numeric_2}Leobacher, G. \& Szölgyenyi, M. A strong order 1/2 method for multidimensional SDEs with discontinuous drift. {\em Ann. Appl. Probab.}. \textbf{27}, 2383-2418 (2017), https://doi.org/10.1214/16-AAP1262

\bibitem{stannat_mehri_lyapunov}Mehri, S. \& Stannat, W. Weak solutions to Vlasov-McKean equations under Lyapunov-type conditions. {\em Stoch. Dyn.}. \textbf{19}, 1950042, 23 (2019), https://doi.org/10.1142/S0219493719500424

\bibitem{mishura_true_mckean_vlasov}Mishura, Y. \& Veretennikov, A. Existence and uniqueness theorems for solutions of Mckean-Vlasov stochastic equations. {\em Theory Probab. Math. Statist.}., 59-101 (2020), https://doi.org/10.1090/tpms/1135

\bibitem{paper1}Nykänen, J. Mean-field stochastic differential equations with a discontinuous diffusion coefficient. {\em Probability, Uncertainty And Quantitative Risk}. \textbf{8}, 351-372 (2023)




\bibitem{disc_drift_numeric_1}Simonsen, M., Leth, J., Schioler, H. \& Cornean, H. A simple stochastic differential equation with discontinuous drift. {\em Proceedings Third International Workshop On Hybrid Autonomous Systems}. \textbf{124} pp. 109-123 (2013), https://doi.org/10.4204/EPTCS.124.11

\bibitem{sarkka}Särkkä, S. \& Solin, A. Applied stochastic differential equations. (Cambridge University Press, Cambridge, 2019)

\bibitem{torres_viitasaari}Torres, S. \& Viitasaari, L. Stochastic differential equations with discontinuous diffusion coefficients. {\em Theory Probab. Math. Statist.}., 159-175 (2023), https://doi.org/10.1090/tpms/1201

\bibitem{zhang_discrete}Zhang, X. A discretized version of Krylov's estimate and its applications. {\em Electron. J. Probab.}. \textbf{24} pp. Paper No. 131, 1-17 (2019), https://doi.org/10.1214/19-ejp390



\end{thebibliography}
\end{document}